\documentclass[12pt]{amsart}
\usepackage[T1]{fontenc}
\usepackage[all]{xypic}
\usepackage[colorlinks,citecolor=blue]{hyperref}
\newtheorem{theorem}[subsection]{Theorem}
\newtheorem{lemma}[subsection]{Lemma}
\newtheorem{proposition}[subsection]{Proposition}
\newtheorem{convention}[subsection]{Convention}

\newtheorem{definition}[subsection]{Definition}

\usepackage{geometry}
\newcommand{\red}[1]{{{#1}}}

\newcommand{\br}{\mathrm{Br}}
\newcommand{\brgp}[1]{\mathrm{Br}(#1)}
\newcommand{\brgpu}[1]{\mathrm{Br}_1(#1)}
\newcommand{\brgpa}[1]{\mathrm{Br}_a(#1)}

\newcommand{\cor}[1]{\mathrm{Cor}_{#1}}
\newcommand{\zcycle}[1]{\mathrm{Z}_0^{#1}}
\newcommand{\zcyclef}[2]{\mathrm{Z}_{0,#2}^{#1}}
\newcommand{\prodf}[1]{\mathop{\prod}\limits_{#1}}
\newcommand{\sumf}[1]{\mathop{\sum}\limits_{#1}}
\newcommand{\inv}{{\mathrm{inv}}}
\newcommand{\reff}[1]{{(\ref{#1})}}
\newcommand{\resp}[1]{(resp. #1)}
\newcommand{\spec}[1]{\mathrm{Spec}(#1)}
\newcommand{\cohomorf}[2]{\mathrm{H}^{#1}(#2)}
\newcommand{\cohomorc}[2]{\mathrm{H}_{c}^{#1}(#2)}
\newcommand{\ET}{{\text{\'etale}}}
\newcommand{\et}{{\text{\'et}}}
\newcommand{\image}[1]{\mathrm{Im}(#1)}
\newcommand{\dsum}{{\oplus}}
\newcommand{\Dsum}{{\bigoplus}}
\newcommand{\Dsumf}[1]{\mathop{\bigoplus}\limits_{#1}}
\newcommand{\kernel}[1]{\mathrm{Ker}(#1)}
\newcommand{\cA}{\mathcal{A}}
\newcommand{\cC}{\mathcal{C}}
\newcommand{\cF}{\mathcal{F}}
\newcommand{\cG}{\mathcal{G}}
\renewcommand{\cH}{\mathcal{H}}
\newcommand{\cO}{\mathcal{O}}

\newcommand{\cT}{\mathcal{T}}

\newcommand{\cX}{\mathcal{X}}

\newcommand{\vA}{\mathbf{A}}
\newcommand{\bC}{\mathbb{C}}
\newcommand{\bG}{\mathbb{G}}
\newcommand{\bQ}{\mathbb{Q}}
\newcommand{\bZ}{\mathbb{Z}}
\newcommand{\dercat}[1]{\mathrm{D}(#1)}
\renewcommand{\hom}[1]{\mathrm{Hom}_{#1}}
\usepackage[OT2,T1]{fontenc}
\DeclareSymbolFont{cyrletters}{OT2}{wncyr}{m}{n}
\DeclareMathSymbol{\Sha}{\mathalpha}{cyrletters}{"58}

\newcommand{\PTcouple}[1]{\langle#1\rangle_{\mathrm{PT}}}
\newcommand{\BMcouple}[1]{\langle#1\rangle_{\mathrm{BM}}}

\begin{document}

	\author{Diego Izquierdo, Yongqi Liang, Hui Zhang}
	\title[Product of Brauer--Manin Obstructions]{\bf{{Product of Brauer--Manin obstructions for 0-cycles over number fields and function fields}}}
	
	\setcounter{secnumdepth}{4}
	
	\maketitle 
	
\begin{abstract}
	It is conjectured that the Brauer--Manin obstruction is expected to control the existence of 0-cycles of degree 1 on smooth proper varieties over number fields. In this paper, we prove that the existence of Brauer--Manin obstruction to Hasse principle for 0-cycles of degree 1 on the product of smooth (non-necessarily proper) varieties is equivalent to the simultaneous existence of such an obstruction on each factor.
	
	\red{We also prove an analogous statement for smooth varieties defined over function fields of $\bC((t))$-curves.}

\end{abstract}

\section{Introduction}
	Consider a number field $k$ or the function field of a projective smooth geometrically integral $\bC((t))$-curve. We will call the latter case the function field case in this article. For simplification, all varieties appearing in this section will be smooth and geometrically integral. 
	
	For a class of $k$-varieties $\cF$, we say that the class $\cF$ satisfies Hasse principle if for each $k$-variety $V\in\cF$, $V(\vA_k)\not=\emptyset$ implies $V(k)\not=\emptyset$. When $V(\vA_k)\not=\emptyset$ and $V(k)=\emptyset$, one may want to explain the failure of Hasse principle. To do so, Manin introduced a subset $V(\vA_k)^\br$ of $V(\vA_k)$. The set $V(\vA_k)^\br$ is defined as the left-kernel of the Brauer--Manin pairing 
	\begin{equation}
		\langle\cdot,\cdot\rangle:V(\vA_k)\times\brgp{V}\rightarrow\bQ/\bZ.
	\end{equation}Here $\brgp{V}=\mathrm{H}_{\et}^2(V,\bG_m)$ is the cohomological Brauer group of $V$. Using the Brauer--Hasse--Noether exact sequecne when $k$ is a number field and by Weil's reciprocity law when $k$ is a function field, one can prove the inclusions
	\begin{equation}
		V(k)\subseteq V(\vA_k)^\br\subseteq V(\vA_k).
	\end{equation}Hence, if $V(\vA_k)^\br$ is empty, then this will explain why $V(k)$ is empty. From this perspective, if for some $k$-variety $V$, we have $V(\vA_k)^\br=\emptyset$ and $V(\vA_k)\not=\emptyset$, then we say that the Brauer--Manin obstruction to Hasse principle on $V$ exists.

	{It is natural to consider the Brauer--Manin obstruction on the product $Z=X\times_k Y$ of two projective varieties.} Skorobogatov and Zarhin, in their paper \cite[Thm. C]{SZ14}, proved that when $k$ is a number field, {one has $Z(\vA_k)^\br=X(\vA_k)^\br\times Y(\vA_k)^\br$. As a direct corollary, one obtains that} the existence of Brauer--Manin obstruction to Hasse principle for rational points on $Z$ is equivalent to the existence of the same obstruction on $X$ and $Y$ at the same time. The {case without projectivity assumption} is due to Chang Lv, see \cite[Thm. 3.1]{Lv20}. \red{In this paper we are going to prove a version of this statement for 0-cycles of degree 1.} 
	
	Following Manin, Colliot-Th{\'e}l{\`e}ne introduced a similar obstruction associated to 0-cycles in his paper \cite{CT95}. The set $V(\vA_k)$ is replaced by adelic 0-cycles on $V$. One can derive a similar Brauer--Manin pairing between adelic 0-cycles and the Brauer group, so that, we have a set $\zcyclef{}{\vA}(V)^\br$, which turns out to be a group, containing $V(\vA_k)^\br$ as a subset. One can also prove that $\zcyclef{}{\vA}(V)^\br$ contains the group $Z_0(V)$ of global 0-cycles of $V$. \red{We also denote by $\zcyclef{r}{\vA}(V)^\br$ the subset consisting of adelic 0-cycles of degree $r$ at each place.} The group $\zcyclef{}{\vA}(V)^\br$ \resp{the subset $\zcyclef{r}{\vA}(V)^\br$} allows one to define a Brauer--Manin obstruction to Hasse principle for 0-cycles \resp{of degree $r$} similarly to the case of rational points. For a $k$-variety $V$, we say that the Brauer--Manin obstruction for 0-cycles on $V$ \resp{of degree $1$} exists if $\zcyclef{}{\vA}(V)^\br=\emptyset$ \resp{if $\zcyclef{1}{\vA}(V)^\br=\emptyset$}.

	Liang considered weak approximation property in his paper \cite[Thm. 3.1]{Liang23}. Roughly speaking, \red{he proved that, for rationally connected varieties $X$ and $Y$ defined over a number field $k$,  the product $Z=X\times_k Y$ satisfies weak approximation with Brauer--Manin obstruction for 0-cycles if and only if both $X$ and $Y$ satisfy the same property.}
	
	\red{In this paper, we consider the existence of Brauer--Manin obstruction to Hasse principle for 0-cycles on the products of varieties. Given a 0-cycle $x$ on $X$ and a 0-cycle $y$ on $Y$, one can determine a 0-cycle $z$ on $Z$ called product of $x$ and $y$. Precisely, if $x$ and $y$ are simply closed points, then $z$ is the 0-cycle assocated to the fintie $k$-scheme  $\spec{k(x)\otimes_kk(y)}$. For general 0-cycles, one extends the product bilinearly and obtains a bilinear map
		\begin{equation}
			\zcycle{}(X)\times\zcycle{}(Y)\rightarrow\zcycle{}(Z).
		\end{equation}Furthermore, this map gives rise to the bilinear map of adelic 0-cycles
		\begin{equation}
			\Phi:\zcyclef{}{\vA}(X)\times\zcyclef{}\vA(Y)\rightarrow\zcyclef{}\vA(Z).
		\end{equation}We denote this product map $\Phi$.
		
		 We will claim that $\Phi$ in fact induces a product map on Brauer--Manin subsets and it turns out to be a section of the map $\zcyclef{1}\vA(Z)\rightarrow\zcyclef{1}\vA(X)\times\zcyclef{1}\vA(Y)$ induced from canonical projections onto each factor. And as an immediate corollary, one obtains that the existence of such an obstruction on the product of varieties is equivalent to the similtaneous existence of such an obstruction on each factor. More precisely, we state our main result as the follows.}

	\begin{theorem} \label{main2}
	Let $k$ be either a number field or a function field, let $X$ and $Y$ be smooth and geometrically integral varieties over $k$ and $Z=X\times_kY$ be the product of $X$ and $Y$, Denote $p$ and $q$ the projections	
	\begin{equation}
		\xymatrix{
			Y&Z\ar[r]^{p}\ar[l]_{q}&X}.
	\end{equation}
	
	Then \red{the section map
		\begin{equation}
			\xymatrix{
				\zcyclef{1}\vA(X)\times\zcyclef{1}\vA(Y)\ar[r]^{\qquad\Phi}&\zcyclef{1}\vA(Z)\ar@{->>}[r] &\zcyclef{1}\vA(X)\times\zcyclef{1}\vA(Y)}
		\end{equation}gives rise to a section map on Brauer--Manin subsets\begin{equation}
		\xymatrix{
			\zcyclef{1}\vA(X)^\br\times\zcyclef{1}\vA(Y)^\br\ar[r]^{\qquad\Phi}&\zcyclef{1}\vA(Z)^\br\ar@{->>}[r] &\zcyclef{1}\vA(X)^\br\times\zcyclef{1}\vA(Y)^\br.}
			\end{equation}}
	
\end{theorem}

\section*{Notations}
	Here are several basic conventions and notations. 
	
	$\bullet$ A \textit{function field} in this article will always mean the function field of a projective smooth geometrically integral $\bC((t))$-curve.
	
	$\bullet$ The base field $k$ is a number field or a function field. In either case, the field $k$ has characteristic $0$ and \red{in function field case, cohomological dimension is always $\leq2$}. In either case, one can consider natural sets of places associated to $k$: in the number field case, one can take the set of all places of $k$, while in the function field case, one can take the set of closed points of the associated curve. We use the notation $\Omega_k$ to denote this set in both cases if there are no confusions.
	
	$\bullet$ If $T$ is a finite subset of $\Omega_k$, in the number field case, $\cO_T$ is the canonical integer ring associated to $S$, while in the function field case and if $T\not=\emptyset$, $\cO_T$ is the global sections of $C\setminus T$ where $C$ is the associated projective curve. Note that in the function field  case, $C\setminus T$ is affine if $T\not=\emptyset$. \red{In this paper, a finite set $T$ of places always occurs when we use limit argument, so it is harmless to assume $T\not=\emptyset$.}
	
	$\bullet$ Varieties over a fixed field $k$ in this article are assumed to be separated schemes of finite type over $k$.
	
	$\bullet$ Given a field $k$ and a $k$-variety $X$, the group $\zcycle{}(X)$ represents the abelian group of 0-cycles on $X$. If $r$ is a given integer, we also use the symbol $\zcycle{r}(X)$ to refer to the 0-cycles on $X$ of degree $r$.

	$\bullet$ Given a number field or a function field $k$, a $k$-variety $X$ and a place $v\in\Omega_k$, the group $\zcyclef{}{v}(X)$ represents the abelian group $\zcycle{}(X_v)$ where $X_v=X\otimes_kk_v$. Furthermore, we use symbol $\zcyclef{}{\vA}(X)$ to represent the subset $\big\{(z_v)_{v\in\Omega_k}\in\prodf{v\in\Omega_k}\zcyclef{}{v}(X):$ $\exists$ finite subset $S$ of $\Omega_k$ and a model $\cX_T$ of $X$ over $\cO_T$ so that for each $v\not\in T$, $z_v$ comes from a relative 0-cycle $\mathcal{Z}_v$ of $\cX_T\otimes_{\cO_T}\cO_v$ over $\cO_v$ (or a formal finite sum of integral sub-schemes of $\cX_T\otimes_{\cO_T}\cO_v$ which is finite over $\cO_v$)$\big\}$. When $X$ is proper, we have $\zcyclef{}{\vA}(X)=\prodf{v\in\Omega_k}\zcyclef{}{v}(X)$. We call elements of $\zcyclef{}{\vA}(X)$ \textit{adelic 0-cycles} on $X$ over $k$.

	$\bullet$ Given a scheme $X$, the Brauer group of $X$ will always be referred to the $\ET$ cohomology $\brgp{X}=\mathrm{H}_\et^2(X,\bG_m)$. If $X$ is a variety over a field $k$, the first Brauer group $\brgpu{X}$ of $X$ will always refer to the kernel of the homomorphism $\brgp{X}\rightarrow\brgp{X^{sep}}$ and the algebraic Brauer group $\brgpa{X}$ of $X$ will always refer to the cokernel of the homomorphism $\brgp{k}\rightarrow\brgpu{X}$.
	
	$\bullet$ Let $k$ be a number field or a function field, if $\cF$ is an object of the derived category $\mathrm{D}(k)$ of $k$, then we define $\Sha^{i}(\cF)=\kernel{\mathrm{H}_\et^i(k,\cF)\rightarrow\prodf{v\in\Omega_k}\mathrm{H}_\et^i(k_v,\cF)}$.

\section{Preliminaries}
	Assume $V$ is a variety over an arbitrary field $k$ and $z\in V$ is a closed point. Given an element $A\in\brgp{V}$, we can define a pairing $(A,z)=\cor{k(z)/k}(A(z))\in\brgp k$ where $k(z)$ is the residue field of $V$ at $z$ and $A(z)$ is the restriction of $A$ in $\brgp{k(z)}$.
	
	We may linearly eztend the second entry of pairing $(A,z)$ and obtain a pairing 
	\begin{center}
		$(\cdot,\cdot):\brgp V\times\zcycle{}(V)\rightarrow\brgp k$ 
	\end{center}by setting $(A,z)=\sum n_i.(A,z_i)$ where $z=\sum n_i.z_i$ is a 0-cycle of $V$ and each $z_i$ is a closed point of $V$.
	
	The degree of a 0-cycle $z=\sum n_i.z_i$ of a variety $V$ relatively to $k$ is 
	\begin{equation}
		\deg{z}=\sum n_i.[k(z_i):k].
	\end{equation} We call $\deg{z}$ the degree of $z$ if $k$ is clear.
	
	Let $f:W\rightarrow V$ be a morphism between two $k$-varieties. Given a 0-cycle $z$ on $W$, we have the direct image $f_*(z)$ of $z$ along $f$: It is the 0-cycle defined by formula that $f_*(\sum n_i.y_i)=\sum n_i.[k(y_i):k(x_i)].f(y_i)$. Notice that direct image does not change the degree of a 0-cycle. With this definition, one can obtain immediately the following formula 
	\begin{equation}\label{push_formula}
		(f^*(A),z)=(A,f_*(z)) 
	\end{equation} where $z$ is 0-cycle on $W$ and $A$ is an element of $\brgp V$.

	Assume $k$ is a number field or a  function field from now on.	We may define a paring 
	\begin{equation} \label{pairing}
		\zcyclef{r}{\vA}(V)\times\brgp{V}\rightarrow\mathbb{Q}/\mathbb{Z}
	\end{equation}by setting $(A,z)=\sumf{v\in\Omega_k}\inv_v(A,z_v)$ where $z=(z_v)\in\zcyclef{}{\vA}(V)$ is an adelic 0-cycle on $V$. Because of the way $\zcyclef{}{\vA}(V)$ is defined, we know that $\sumf{v\in\Omega_k}\inv_v(A,z_v)$ is a finite sum, hence, it is well-defined. We denote $\zcyclef{}\vA(V)^\br$ the left-kernel of the paring \reff{pairing} and $\zcyclef{r}\vA(V)^\br$ the subset of $\zcyclef{}\vA(V)^\br$ given by the adelic 0-cycles on $V$ of constant degree $r$. We will show later that the diagonal image of $\zcycle{r}(V)$ in $\prodf{v\in\Omega_k}\zcyclef{r}v(V)$ is contained in $\zcyclef{r}\vA(V)^\br$. See lemma \reff{lemma1}.
	
	\red{We say the Brauer--Manin obstruction for 0-cycles on $V$ \resp{of degree $r$} exists if $\zcyclef{}\vA(V)^\br=\emptyset$ \resp{$\zcyclef{r}\vA(V)^\br=\emptyset$}.}

	With these notations, we are ready to prove the theorem \reff{main2}. To do so, we will proceed in several steps. In section 3, we will prove the existence of universal $n$-torsors under the assumptions of the theorem. With the help of universal $n$-torsors, we will give a decomposition of the $n$-torsion part of the Brauer group in section 4 and at the end of section 4, we will finish the proof of the theorem with the help of this decomposition.
	
	We finish this section with following immediate  facts and claim the easy part of the theorem.

	\begin{lemma}\label{lemma1}
		The diagonal (injective) image of $\zcycle{r}(X)$ in $\prodf{v\in\Omega_k}\zcyclef{r}v(X)$ is contained in $\zcyclef{r}\vA(X)^\br$.
	\end{lemma}
	
	\begin{proof}
		It is clear that for each field extension $K/k$, the restriction homomorphism $\zcycle{r}(X)\rightarrow\zcycle{r}(X\otimes_kK)$ is injective, therefore, the diagonal homomorphism $\zcycle{r}(X)\rightarrow\prodf{v\in\Omega_k}\zcyclef{r}v(X)$ is injective. 
		Take $z=\sum n_i.x_i\in\zcycle{r}(X)$ and consider the commutative diagram
		\begin{equation}
			\xymatrix{
				\brgp{X}\ar[r] \ar[d]&\brgp{X_v}\ar[d]\\
				\brgp{k(x_i)}\ar[r] \ar[d]^{\text{cor.}}&\brgp{k(x_i)\otimes_kk_v}\ar[d]^{\text{cor.}}\\
				\brgp{k}\ar[r]_{\psi_v} &\brgp{k_v}}
		\end{equation}
		Let $z_v$ \resp{$x_{i,v}$} be the image of $z$ \resp{$x_i$} in $\zcyclef{r}{v}(X)$, then for each $A\in\brgp{X}$, $(A,z_v)=\sumf{i} n_i.(A,x_{i,v})=\sumf{i} n_i.\psi_v((A,x_i))=\psi_v((A,z))$ and \begin{center}
			$\sumf{v}\inv_v (A,z_v)=\sumf{v}\inv_v\Big(\psi_v((A,z))\Big)=0$
		\end{center}according to the exact sequence \begin{equation}\label{BHN}
			\xymatrix{\brgp k\ar[r]&\Dsum\brgp{k_v}\ar[r]&\bQ/\bZ\ar[r]&0}.
		\end{equation}For the number field case, it is also left-exact due to the
		Brauer--Hasse--Noether exact sequence. For the function field case, proof can be found in \cite[Prop. 2.1 (v)]{Har15}.
	\end{proof}
	
	
	
	\begin{proposition} \label{cor_1}
		Let $f:Y\rightarrow X$ be a morphism of two $k$-varieties, take an adelic 0-cycles $y\in\zcyclef{r}\vA(Y)^\br$, let $x=f_*(y)\in\zcyclef{r}\vA(X)$ be the direct image of $y$, then $x\in\zcyclef{r}\vA(X)^\br$. 
	\end{proposition}
	
	\begin{proof}
		It is enough to notice that for each $A\in\brgp{X}$, $\sum\inv_v(A,x_v)=\sum\inv_v(f^*(A),y_v)=0$ by applying the formula \reff{push_formula}.
	\end{proof}
	
\section{Existence of universal $n$-torsors}
	In this section, we will introduce the concept of universal $n$-torsor and give a sufficient condition for the existence of universal $n$-torsors.
	
	\begin{convention}
		In this article, we only deal with $\ET$ cohomology. So, sheaves always refer to $\ET$ sheaves and instead of symbol $\mathrm{H}_{\et}^i(X,\mathcal{F})$, we denote $\cohomorf{i}{X,\mathcal{F}}$ for convenience. If $\cF^\bullet$ is a complex of $\ET$ sheaves, then $\cohomorf i{X,\cF^\bullet}$ always refers to the hyper-cohomology of $\cF^\bullet$. 
	\end{convention}
	
	Let $S$ be an arbitrary scheme, $X$ be a scheme over $S$ and $\pi$ be the structural morphism $X\xrightarrow{\pi}S$. Let $\Delta(X)$ and $\Delta_n(X)$ be the respective mapping cones of the morphisms 
	\begin{equation}\label{definition_delta}
		\bG_{m,S}[1]\rightarrow\tau_{\leq1}R\pi_*(\bG_{m,\cX})[1]
	\end{equation}\begin{equation}
		\mu_{n,S}[1]\rightarrow\tau_{\leq1}R\pi_*(\mu_{n,\cX})[1].
	\end{equation}If $X$ and $S$ are clear, we denote $\Delta$ and $\Delta_n$ for convenience. 
	
	
	If $k$ is a number field or a function field, we have that $\Sha^1(\Delta)$ is the kernel of the homomorphism
	\begin{equation}
		\xymatrix{\brgpa{X}\ar[r]&\prodf{v\in\Omega_k}\brgpa{X_v}}.
	\end{equation}In fact, this follows from the fact that $\cohomorf1{k,\Delta}$ fits into an exact sequence\begin{equation}
	\xymatrix{\brgp{k}\ar[r]&\brgp{X}\ar[r]&\cohomorf1{k,\Delta}\ar[r]&\cohomorf3{k,\bG_{m,k}}}
	\end{equation}and $\cohomorf3{k,\bG_{m,k}}=0$ is a well-known fact. (For $k$ a number field or a function field, see \cite[Cor. I.4.21]{ADT} or \cite[Prop. 2.1. (iii)]{Har15} respectively.) Then we have diagram with two exact rows
	\begin{equation}
		\xymatrix{
			0\ar[r]&\brgpa{X}\ar[r]\ar[d]&\cohomorf1{k,\Delta}\ar[r]\ar[d]&0\ar[r]\ar[d]&0\\
			0\ar[r]&\prod\brgpa{X_v}\ar[r]^{\inv}&\prod\cohomorf1{k_v,\Delta}\ar[r]&\prod\cohomorf3{k_v,\bG_{m,k_v}}}
	\end{equation}We conclude the result by the snake lemma.
	
	\begin{lemma} \label{lem_harari}
		Let $S$ be an integral regular noetherian scheme, $X$ be an faithfully flat $S$-scheme of finite type, $G$ be a $S$-group scheme of multiplicative type, then there exists an exact sequence 
		\begin{equation}\label{complex_harari}
			\xymatrix{\cohomorf{1}{S,G}\ar[r]&\cohomorf{1}{X,G}\ar[r]^-{\chi}&\mathrm{Hom}_{\mathrm{D}(S)}(\hat{G},\Delta)\ar[r]^-{\partial}&\cohomorf{2}{S,G}\ar[r]&\cohomorf{2}{X,G}}
		\end{equation}where $\hat{G}$ is the Cartier dual of $G$ and $\mathrm{D}(S)$ is the derived category of sheaves over $S$.
	\end{lemma}
	
	\begin{proof}
		See \cite[Prop. 1.1]{Har13}. Note that in \textit{loc. cit}, $\partial$ is induced by the canonical homomorphism \begin{equation}
			\xymatrix{R^0\mathrm{Hom}_{\mathrm{D(S)}}(\hat{G},\Delta)\ar[r]&R^2\mathrm{Hom}_{\mathrm{D(S)}}(\hat{G},\bG_m)\approx\cohomorf2{S,G}}.
		\end{equation}	
	\end{proof}
	
	
	\begin{definition}
		Let $G$ be the $k$-group variety of multiplicative type whose Cartier dual (explain) is the constant $k$-group variety corresponding to the finite group $\cohomorf{1}{X^{sep},\mu_n}$, there is a canonical morphism $\psi:\hat{G}=\Delta_n\rightarrow\Delta$ in $\mathrm{D}(k)$. We call any pre-image $\mathcal{T}_X\in\cohomorf1{X,G}$ of $\psi$ along $\chi$ (denoted in lemma \reff{lem_harari}) a universal $n$-torsor of $X$.
		
		Note that for a universal $n$-torsor to exist, "$\partial$ is the zero homomorphism" is a sufficient condition.
	\end{definition}
	
	Before we provide a sufficient condition for the existence of universal $n$-torsors, we need prove some lemmas.
	
	\begin{definition}\label{def_inv}
%
		Assume that $k$ is a number field or a function field, $X$ is a variety over $k$ so that $\zcyclef{1}{\vA}(X)\not=\emptyset$. Consider the following commutative diagram
		\begin{equation}
		\xymatrix{
				&&&\Sha^1(\Delta)\ar[d]&\\
				&\brgp{k}\ar[r]\ar[d]&\brgpu{X}\ar[r]\ar[d]&\brgpa{X}\ar[r]\ar[d]&0\\
				0\ar[r]&\prod\brgp{k_v}\ar[r]\ar[d]^{\inv}&\prod\brgpu{X_v}\ar[r]&\prod\brgpa{X_v}\ar[r]&0\\
				&Q}
		\end{equation}Here $Q$ is the cokernel of homomorphism $\brgp k\rightarrow\prodf{v\in\Omega_k}\brgp{k_v}$ and $\brgp{k_v}\rightarrow\brgpu{X_v}$ is injective due to the existence of retraction induced by an adelic 0-cycle of degree 1. According to the exact sequence \reff{BHN}, we may treat $\bQ/\bZ$ as a subgroup of $Q$. By the snake lemma, we have a homomorphism $\Sha^1(\Delta)\rightarrow Q$. In the following lemma, we will show that the image of this homomorphism is indeed contained in $\bQ/\bZ$. Thus, it defines a homomorphism $\Sha^1(\Delta)\rightarrow\bQ/\bZ$. We denote this homomomorphism $\inv$ as well.

	\end{definition}
	
	
	\begin{lemma} \label{lemma_welldefined}
		Keeping the notations in definition \reff{def_inv}, assume that $X$ is smooth and geometrically integral and $\zcyclef{1}{\vA}(X)\not=\emptyset$, then the homomorphism $\Sha^1(\Delta)\xrightarrow{\inv}\bQ/\bZ$ is well-defined.
	\end{lemma}

	\begin{proof}
		Fix an adelic 0-cycle  $z=(z_v)\in\zcyclef{1}{\vA}(X)$ on $X$.  There is a non-empty finite subset $T$ of $\Omega_k$ so that there is a smooth model $\cX$ of $X$ over $\cO_T$. By the definition of adelic 0-cycles, we may enlarge $T$ so that for each $v\not\in T$, $z_v=\sum n_{v,i}.x_{v,i}$ comes from a finite formal sum $\sum n_{v,i}.R_{v,i}$ where $R_{v,i}$ is integral finite $\cO_v$-algebra, hence, free over $\cO_v$ with rank $[k(x_{v,i}):k_v]$. Notice that such a formal sum also gives a retraction $\brgp\cX\rightarrow\brgp{\cO_v}$ by corestrictions and it is compatible with the retraction $s_v$ given by $z_v$. Namely, we have a commutative diagram
		
		\begin{equation}\label{18}
			\xymatrix{
				\brgp{\cX}\ar[r]\ar[d]&\brgp{\cX\otimes_{\cO_T}\cO_v}\ar[r]\ar[d]&\brgp{\cO_v}=0\ar[d]\\
				\brgp{X}\ar[r]&\brgp{X_v}\ar[r]_{s_v}&\brgp{k_v}.}
		\end{equation}

%
%
		
		
		By the definition of the homomorphism $\Sha^1(\Delta)\rightarrow Q$, given $\overline\alpha\in\Sha^1(\Delta)$, we take a pre-image $\alpha\in\brgpu{X}$ of $\overline\alpha$. By enlarging $T$, we may assume that $\alpha$ admits a lift $\alpha_T$ in $\brgp{\cX}$. Since the restriction $\alpha_v\in\brgpu{X_v}$ is mapped to zero in $\brgpa{X_v}$, we have $\alpha_v\in\brgp{k_v}$. But for each $v\not\in T$, $\alpha_v$ coincide the image of $\alpha_T$ along the path from the left-up corner of diagram \reff{18} to the right-down corner, hence, we have $\alpha_v=0$. Thus, $(\alpha_v)_{v\in\Omega_k}\in\Dsum\brgp{k_v}$ and the invariant of $\overline\alpha$ lies in $\bQ/\bZ$.
	\end{proof}
	
	Next, we will discuss the relation between Poitou--Tate pairing and the invariant homomorphism above.
	
	Let $G$ be a $k$-group of multiplicative type, there is a perfect Poitou--Tate pairing of finite groups. 
	\begin{equation}
		\langle\,,\,\rangle_{\text{PT}}:\Sha^2(G)\times\Sha^1(\hat G)\rightarrow\bQ/\bZ.
	\end{equation}For $k$ a number field, see \cite[Thm. 5.7]{Der11}. For $k$ a function field, see \cite[Thm. 2.4]{Izq16}. It is defined as follows. Let $a\in\Sha^1(\hat{G})$ and $b\in\Sha^2(G)$ and $\cG$ be a smooth model of $G$ over a non-empty open subset $W$ of $\spec\bZ$ or the $\bC((t))$-curve $C$ associated to $k$. By a limit argument, $\cohomorf2{k,G}$ is the direct limit of the groups $\cohomorf2{U,\cG}$ where $U$ runs over non-empty open subsets of $W$. For $U$ sufficiently small, we can lift $a$ to $a'_U\in\cohomorf1{U,\hat\cG}$ where $\hat\cG=\text{Hom}(\cG,\bG_{m,W})$ and $b$ to $b_U\in\cohomorf2{U,\cG}$. For any object $\cC$ of $\text{D}(U)$, we have hyper-cohomology groups with compact support $\cohomorc i{U,\cC}$. (For $k$ a function field, it is just $\cohomorf i{U,j_!(\cC)}$. But for $k$ a number field, it is more subtle. See \cite[Section 3]{Har05} for the detailed definition.) By localization exact sequence 
	\begin{equation}
		\cdots\rightarrow\cohomorc i{U,\cC}\rightarrow\cohomorf i{U,\cC}\rightarrow\Dsumf{v\not\in U}\cohomorf i{\hat k_v,\cC}\rightarrow\cdots,
	\end{equation}since $a$ is locally trivial everywhere, $a'_U$ comes from some $a_U\in\cohomorc 1{U,\hat\cG}$ under the natural map $\cohomorc 1{U,\hat\cG}\rightarrow\cohomorf1{U,\hat\cG}$. We define $\langle b,a\rangle_{\text{PT}}$ as the cup-product $b_U\cup a_U\in\cohomorc3{U,\bG_{m,U}}\approx\bQ/\bZ$ where the last isomorphism comes from the trace map, see \cite[Prop. II.2.6]{ADT} for the number field case or \cite[Prop. 2.1 (iii)]{Har15} for the function field case. We don't know whether this definition of the Poitou--Tate pairing coincides with the classical definition in terms of cocycles, but we shall only use the fact that it leads to a perfect paring.

	\begin{lemma} \label{lemma_arith}
		Let $X$ be a smooth and geometrically integral variety over a number field or function field $k$, $G$ be $k$-group variety of multiplicative type. Assume that for each $v\in\Omega_k$, canonical homomorphism
		\begin{equation}
			\xymatrix{\cohomorf2{k_v,G}\ar[r]&\cohomorf2{X_v,G}}
		\end{equation}is injective. Let $\psi\in\mathrm{Hom}_{\mathrm{D}(k)}(\hat{G},\Delta)$ be a homomorphsim in the derived category of $k$ and $A\in\Sha^1(\hat G)$, then $\partial(\psi)\in\Sha^2(G)$ and we have \begin{equation}
		\PTcouple{\partial(\psi),A}=\inv(\psi_*(A))
		\end{equation}where $\psi_*$ is the induced homomorphism $\cohomorf1{k,\hat{G}}\rightarrow\cohomorf1{k,\Delta}$ and $\partial$ is the homomorphism appearing in lemma \reff{lem_harari}.
	\end{lemma}
	
	\begin{proof}
		This proof follows from the proof of \cite[Prop. 3.5]{Har13}.
		
	{We only claim the case of number field. One can obtain the proof of the case of function field by simply replacing $\spec\bZ$ by the associated $\bC((t))$-curve $C$.}
		
	Note that the image of $\partial(\psi)$ in $\cohomorf2{X,G}$ is zero by lemma \reff{lem_harari}. Since $\cohomorf2{k_v,G}\rightarrow\cohomorf2{X_v,G}$ is injective by assumption, the image of $\partial(\psi)$ in $\prod\cohomorf2{k_v,G}$ is zero. Therefore, we have $\partial(\psi)\in\Sha^2(G)$.
	
	Let $w$ be the canonical homomorphism $\Delta\rightarrow \bG_{m,k}[2]$ by definition of $\Delta$ and $\pi$ be the structural morphism $X\xrightarrow\pi \spec k$. Since exact sequence \reff{complex_harari} is obtained by applying the functor $\hom{\dercat{k}}(\hat G,\ast)$ to distinguished triangle \begin{equation}
		\xymatrix{\Delta\ar[r]^w&\bG_{m,k}[2]\ar[r]&\tau_{\leq1}R\pi_*(\bG_{m,X})[2]\ar[r]&[1]}
	\end{equation}under the isomorphism $\hom{\dercat{k}}(\hat G,\bG_{m,k}[2])=\cohomorf2{k,G}$, we have $w\circ\psi=\partial(\psi)$. 

	Let $U\subseteq\spec{\cO_k}$ be a sufficiently small non-empty open subset so that there exists a smooth $U$-scheme $\cX$ with geometrically integral fibres and the generic fibre $X=\cX\times_Uk$, and a smooth $U$-group of multiplicative type $\cG$ with the generic fibre $G=\cG\times_Uk$. Let $w_U\in\hom{\dercat{U}}(\Delta(\cX),\bG_{m,U}[2])$ be defined in the same manner of the definition of $w$. There is a canonical restriction homomorphism \begin{equation}
	\xymatrix{\hom{\dercat{U}}(\hat\cG,\Delta(\cX))\ar[r]&\hom{\dercat{U}}(\hat G,\Delta(X))}
	\end{equation}and for each $V\subseteq U$ a non-empty open subset, we obtain a commutative diagram
	\begin{equation}\xymatrix{\cohomorf1{\cX_V,\cG_V}\ar[r]\ar[d]&\hom{\dercat{U}}(\hat\cG,\Delta(\cX))\ar[r]\ar[d]&\cohomorf2{V,\cG_V}\ar[r]\ar[d]&\cohomorf2{\cX_V,\cG_V}\ar[d]\\\cohomorf1{X,G}\ar[r]&\hom{\dercat{k}}(\hat G,\Delta(X))\ar[r]&\cohomorf2{k,G}\ar[r]&\cohomorf2{X,G}.}
	\end{equation}Note that the rows of this diagram are exact by lemma \reff{lem_harari}. Passing to the limit, one obtains an isomorphism
	\begin{equation}
		\xymatrix{\mathop{\lim}\limits_{\xrightarrow[V]{}}\hom{\dercat{V}}(\hat\cG,\Delta(\cX))\ar[r]&\hom{\dercat{k}}(\hat G,\Delta(X)).}
		\end{equation}Hence, after shrinking $U$, there is a lift $\psi_U\in\hom{\dercat{U}}(\hat\cG,\Delta(\cX))$ of $\psi$ and $w_U\circ\psi_U\in\hom{\dercat{U}}(\hat\cG,\bG_{m,U}[2])=\cohomorf2{U,\cG}$ is a lift of $w\circ\psi=\partial(\psi)$.
	
	Denote $\alpha=\psi_*(A)\in\Sha^1(\Delta)$. By shrinking $U$ further, we can find a lift $A_U\in\cohomorf1{U,\hat\cG}$ of $A$. Since $A$ vanishes locally everywhere, we have $A_U\in\cohomorc1{U,\hat\cG}$. Let $\alpha_U=\psi_{U\ast}(A_U)$, then $\alpha_U$ is a lift of $\alpha$. 
	
	Consider the commutative diagram of distinguished triangles\begin{equation}
		\xymatrix{\bG_{m,U}[1]\ar[r]\ar[d]&\tau_{\leq1}Rp_{U*}(\bG_{m,\cX})[1]\ar[r]\ar[d]&\Delta(\cX)\ar[d]\\
		\Dsumf{v\not\in U} j_{v*}j^*_v\bG_{m,U}[1]\ar[r]&\Dsumf{v\not\in U} j_{v*}j^*_v\tau_{\leq1}Rp_{U*}(\bG_{m,\cX})[1]\ar[r]&\Dsumf{v\not\in U} j_{v*}j^*_v\Delta(\cX)}
		\end{equation}Let $\cC_0,\cC_1,\cC_2$ be the mapping cones of three vertical homomorphisms, respectively. The group $\cohomorf1{U,\cC_2[-1]}$ may be identified with $\cohomorc1{U,\Delta(\cX)}$. So we may identify $\alpha_U$ as an element of $\cohomorf1{U,\cC_2[-1]}$. Therefore, $w_{U*}(\alpha_U)$ is an element of $\cohomorf2{U,\cC_0[-1]}=\cohomorf1{U,\cC_0}$ and $\inv(\alpha)=w_{U*}(\alpha_U)$ by snake lemma and passing $U$ to the limit.
	
	Therefore, we have\begin{equation}
	\inv(\psi_*(A))=\inv(\alpha)=w_{U*}(\alpha_U)=w_{U*}(\psi_{U*}(A_U))=(w_{U*}\circ\psi_{U})\cup A_U
	\end{equation}and $(w_{U*}\circ\psi_{U})\cup A_U$ is exactly the definition of $\PTcouple{\partial(\psi),A}$. 
	\end{proof}

	\begin{proposition} \label{prop_exist}
		Let $k$ be a number field or function field and $X$ be a smooth and geometrically integral variety over $k$, let $n\geq2$ be a number and $G$ be the Cartier dual of $\cohomorf{1}{X^{sep},\mu_n}$. If $\zcyclef{1}\vA(X)^\br\not=\emptyset$, then $\chi$ is surjective where $\chi$ is the homomorphism defined in lemma \reff{lem_harari}. Hence, universal $n$-torsors on $X$ always exist. 

	\end{proposition}

	\begin{proof}
		Notice that if $\zcyclef{1}\vA(X)^\br\not=\emptyset$, then the hypothesis of lemma \reff{lemma_arith} is satisfied. We only need to show that $\partial(\psi)=0$ for each $\psi\in\mathrm{Hom}_{\mathrm{D}(k)}(\hat{G},\Delta)$ by exact sequence \reff{complex_harari}. By Poitou--Tate exact sequence, we only need to show that for each $A\in\Sha^1(\hat S)$, $\PTcouple{\partial(\psi),A}=0$. By lemma \reff{lemma_arith}, we have $\PTcouple{\partial(\psi),A}=\inv(\psi_*(A))$. Recall that $\psi_*(A)$ lies in $\Sha^1(\Delta)\subseteq\brgpa{X}$, we take $\alpha$ as a pre-image of $\psi_*(A)$ in $\brgpu{X}$. For each $v\in\Omega_k$, let $s_v$ be a retractions of $\brgp{k_v}\rightarrow\brgpu{X_v}$ defined by some adelic 0-cycle of degree $1$ on $X$. Then we have $\inv(\psi_*(A))=\sum\inv_v(s_v(\alpha(X_v)))=0$ by snake lemma. (Here, the sum is finite due to the same reason as in the proof of lemma \reff{lemma_welldefined}) 
	\end{proof}

\section{A Decomposition of $\cohomorf{2}{Z,\mu_n}$}
	In this section, we will give a decomposition of Brauer group of $Z$ and then use it to prove Theorem \reff{main2}.

	\begin{proposition} \label{prop_decomposition}
		Let $k$ be a number field or a function field, $X$ and $Y$ are two smooth and geometrically integral varieties $k$-varieties, and $Z=X\times_k Y$ is the product of $X$ and $Y$. Denote canonical arrows between them as follows:	
				\begin{equation}
				\xymatrix{
								Z\ar[r]^{p} \ar[d]_{q}&X\ar[d]^{q_0}\\
								Y\ar[r]_{p_0}&\spec{k}.}
		\end{equation}Assume that $\zcyclef{1}{\vA}(X)\not=\emptyset$ and $\zcyclef{1}{\vA}(Y)\not=\emptyset$. Let $\cT_X$ and $\cT_Y$ be universal $n$-torsors on $X$ and $Y$ respectively. The existence of such universal $n$-torsors is guaranteed by \reff{prop_exist}. Let $G_X$ \resp{$G_Y$} be the Cartier  dual of $\cohomorf1{X^{sep},\mu_n}$ \resp{$\cohomorf1{Y^{sep},\mu_n}$}. Then we have the following decomposition
		\begin{equation}
			\cohomorf2{Z,\mu_n}=p^*\cohomorf2{X,\mu_n}+q^*\cohomorf2{Y,\mu_n}+\image{\epsilon}
		\end{equation} where $\epsilon$ is the homomorphism $\hom{k}(G_X\otimes G_Y,\mu_n)\rightarrow\cohomorf2{Z,\mu_n}$ by sending $\psi\in\hom{k}(G_X\otimes G_Y,\mu_n)$ to $\epsilon(\psi)=\psi_*(\cT_X\cup\cT_Y)$.
	\end{proposition}
	
	\begin{proof}
		Let $\Gamma$ be the absolute Galois group of $k$. Recall the spectral sequence 
		\begin{equation}
			E_2^{i,j}=\cohomorf i{k,\cohomorf{j}{X^{sep},\mu_n}}\Rightarrow\cohomorf{i+j}{X,\mu_n}.
		\end{equation}We will focus on the line $i+j=2$ of this spectral sequence:
		\begin{equation}\label{ss_diag}
			\xymatrix{\cohomorf2{X^{sep},\mu_n}^\Gamma\ar[rrd]^{\delta_X}&\ast&\ast&\ast\\
			\cohomorf{1}{X^{sep},\mu_n}^\Gamma\ar[rrd]&\cohomorf1{k,\cohomorf{1}{X^{sep},\mu_n}}\ar[rrd]&\cohomorf2{k,\cohomorf{1}{X^{sep},\mu_n}}&\ast\\
			\ast&\ast&\cohomorf2{k,\mu_n}&\cohomorf3{k,\mu_n}.}
		\end{equation}Firstly, we claim that $\cohomorf i{k,\mu_n}\rightarrow\cohomorf i{X,\mu_n}$ is injective for $i\geq3$. In fact, in the case of function field, it is the direct consequence of the fact that $\text{cd}(k)\leq2$. In the case of number field, $\cohomorf i{k,\mu_n}\rightarrow\Dsumf{v\text{ real}}\cohomorf i{k_v,\mu_n}$ is bijective for $i\geq3$, see \cite[Thm. I.4.10 (c)]{ADT}. Since any adelic 0-cycle of degree 1 defines a retraction of the homomorphism $\cohomorf i{k_v,\mu_n}\rightarrow\cohomorf i{X_v,\mu_n}$ for any $p$ or any $v\in\Omega_k$, it is injective. Hence, 
		\begin{equation}
			\cohomorf i{k,\mu_n}\rightarrow\cohomorf i{X,\mu_n}\rightarrow\Dsumf{v\text{ real}}\cohomorf i{X_v,\mu_n}
		\end{equation}is injective and this implies the claim. With this claim, we obtain the triviality of all the homomorphisms in the spectral sequence \reff{ss_diag} whose target is $E^{i,0}_2=\cohomorf i{k,\mu_n}$ for $i\geq3$.
		
		Since the line $i+j=2$ is stable after page $3$ using the above claim, we obtain an exact sequence for $\cohomorf{2}{X,\mu_n}$:
		\begin{equation}
			\begin{array}{c}
			\xymatrix{
			0\ar[r]&E^{2,0}_3\ar[r]&\cohomorf2{X,\mu_n}\ar[r]&A\ar[r]&0}\\
			\xymatrix{0\ar[r]&E^{1,1}_3\ar[r]&A\ar[r]&E^{0,2}_3\ar[r]&0}
			\end{array}
		\end{equation}where $E^{2,0}_3=\cohomorf2{k,\mu_n}/\image{\cohomorf1{X^{sep},\mu_n}^\Gamma},E^{1,1}_3=\cohomorf1{k,\cohomorf1{X^{sep},\mu_n}},E^{0,2}_3=\kernel{\delta_X}$ where $\delta_X$ is the homomoprhism in diagram \reff{ss_diag}. Therefore, we obtain \newline $A=\cohomorf2{X,\mu_n}/\image{\cohomorf2{k,\mu_n}}$ and an exact sequence 
		\begin{equation}\label{ss_1}
				\xymatrix{0\ar[r]&\cohomorf1{k,\cohomorf{1}{X^{sep},\mu_n}}\ar[r]&\cohomorf2{X,\mu_n}/\image{\cohomorf2{k,\mu_n}}\ar[r]&\kernel{\delta_X}\ar[r]&0}.
		\end{equation} There are similar sequences for $Y$ and $Z$ linked by the homomorphisms $p^*$ and $q^*$.
		
		Let us denote
		\begin{equation}
			\cH=p^*\cohomorf2{X,\mu_n}+q^*\cohomorf2{Y,\mu_n}+\image{\epsilon}\subseteq\cohomorf2{Z,\mu_n}.
		\end{equation} 
		Note that the image of $\cohomorf2{k,\mu_n}$ in $\cohomorf2{Z,\mu_n}$ is contained in $\cH$. It is enough to prove that the canonical homomorphism $\cH\rightarrow\cohomorf2{Z,\mu_n}/\image{\cohomorf2{k,\mu_n}}$ is surjective.
		
		By \cite[Prop. 2.6]{Cao24}, we have an isomorphism \begin{equation}
			\cohomorf1{Z^{sep},\mu_n}\approx p^*\cohomorf1{X^{sep},\mu_n}\dsum q^*\cohomorf1{Y^{sep},\mu_n}
		\end{equation}This implies that the image of $\cohomorf1{k,\cohomorf1{Z^{sep},\mu_n}}$ in $\cohomorf2{Z,\mu_n}/\image{\cohomorf2{k,\mu_n}}$ is contained in $\cH$. By the exact sequence \reff{ss_1}, it remains to claim that each element of $\kernel{\delta_Z}$ admits a pre-image from $\cH$.
		
		By \cite[Cor. 2.7]{Cao24} and Galois descent, we have an isomorphism of $\Gamma$-modules
		\begin{equation}
			\cohomorf2{Z^{sep},\mu_n}^\Gamma\approx\cohomorf2{X^{sep},\mu_n}^\Gamma\dsum\cohomorf2{Y^{sep},\mu_n}^\Gamma\dsum\hom k(G_X\otimes G_Y,\mu_n)
		\end{equation}Each $\psi\in\hom{k}(G_X\otimes G_Y,\mu_n)$, considered as an element of $\cohomorf2{Z^{sep},\mu_n}^\Gamma$, lifts to $\psi_*(\cT_X\cup\cT_Y)\in\cohomorf2{Z,\mu_n}$ in \textit{loc. cit}. Hence $\delta_Z$ is zero on the direct summand $\hom{k}(G_X\otimes G_Y,\mu_n)$ of $\cohomorf2{Z^{sep},\mu_n}^\Gamma$ by \reff{ss_1}, so that $\delta_Z$ is the direct sum of homomorphisms
		\begin{equation}
			\begin{array}{c}
				\xymatrix{\delta_X:\cohomorf2{X^{sep},\mu_n}^\Gamma\ar[r]&\cohomorf2{k,\cohomorf1{X^{sep},\mu_n}}}\\
				\xymatrix{\delta_Y:\cohomorf2{Y^{sep},\mu_n}^\Gamma\ar[r]&\cohomorf2{k,\cohomorf1{Y^{sep},\mu_n}}}\\
				\xymatrix{\hom{k}(G_X\otimes G_Y,\mu_n)\ar[r]&0}
			\end{array}
		\end{equation}Thus $\kernel{\delta_Z}$ is the surjective image of $\cohomorf2{X,\mu_n}\dsum\cohomorf2{Y,\mu_n}\dsum\image{\epsilon}$ by sequence \reff{ss_1}.
	\end{proof} 
	
	Before we prove the theorem \reff{main2}, we need calculate some Brauer-Manin pairing at first. We begin with the following basis algebraic lemma.
	
	\begin{lemma} \label{arith2}
		Under the same hypothesis as the proposition \reff{prop_decomposition}, take $x\in\zcyclef{1}\vA(X)$,  $y\in\zcyclef{1}\vA(Y)$ and  $\psi\in\hom k(G_X\otimes G_Y,\mu_n)$, let $z=\Phi(x,y)$ be the product of adelic 0-cycles $x$ and $y$, then we have
		\begin{equation} \label{changement}
			\BMcouple{\epsilon(\psi),z}=\PTcouple{(\psi_*(\cT_X),x),(\cT_Y,y)}
		\end{equation}where $\psi_*$ is the induced homomoprhism $\cohomorf{2}{X,G_X}\rightarrow\cohomorf2{X,\hat G_Y}$  
	\end{lemma}
	
	\begin{proof}
		First consider that $x$ and $y$ are global 0-cycles on $X$ and $Y$. Since the equation \reff{changement} is bilinear in entry $x$ and $y$, we may assume $x$ and $y$ are single closed points. Since $k(x)$ is separable extension of $k$, $z$ correponds formal sum of closed points of $\spec{k(x)\otimes_kk(y)}$. Write $z=\sum z_i$ and denote $a$ the arrow $\spec{k(y)}\xrightarrow a\spec{k}$. Recall that $\epsilon(\psi)$ is the image of $\cT_X\cup\cT_Y=p^*(\cT_X)\cup q^*(\cT_Y)$  along $\cohomorf{2}{Z,G_X\otimes G_Y}\rightarrow\cohomorf2{Z,\mu_n}$ and we have $\epsilon(\psi)=p^*(\psi_*(\cT_X))\cup q^*(\cT_Y)$. Denote $\cA=\psi_*(\cT_X)$ for convenience, then we have \begin{equation}
			\cor{k(z_i)/k}(\cA(z_i)\cup\cT_Y(z_i))=\cor{k(y)/k}(\cor{k(z_i)/k(y)}(\cA(z_i))\cup\cT_Y(y))
		\end{equation}
		
		Therefore, 
			\begin{align}
			\sumf i\cor{k(z_i)/k}(\cA(z_i)\cup\cT_Y(z_i))&=\cor{k(y)/k}(\sumf i\cor{k(z_i)/k(y)}(\cA(z_i))\cup\cT_Y(y))\\
				&=\cor{k(y)/k}(a^*(\cor{k(x)/k}(\cA(x)))\cup\cT_Y(y))\\
				&=\cor{k(x)/k}(\cA(x))\cup\cor{k(y)/k}(\cT_Y(y))\\
				&=(\cA,x)\cup(\cT_Y,y).\label{a}
			\end{align}
			
			For adelic 0-cycles, with formula \reff{a}, we have
			\begin{equation}
				\BMcouple{\epsilon(\psi),z}=\sumf{i}\sumf{v\in\Omega_v}\inv_v((\cA,x_v)\cup(\cT_Y,y_v))=\PTcouple{(\psi_*(\cT_X),x),(\cT_Y,y)}.
			\end{equation}  
	\end{proof}
	
	We are  now ready to prove Theorem \reff{main2}.

	\textit{Proof of }\textbf{Theorem \reff{main2}}. 
		According to corollary \reff{cor_1}, we only need to claim that if $x\in\zcyclef{1}{\vA}(X)^\br$ and $y\in\zcyclef{1}{\vA}(Y)^\br$, then $z=\Phi(x,y)$ is contained in $\zcyclef{1}{\vA}(Z)^\br$.
		
		Since the Brauer group of a smooth variety is a torsion group, and according to proposition \reff{prop_decomposition}, for each positive integer $n$, the group $\brgp{Z}[n]=\cohomorf2{Z,\mu_n}$ is generated by $p^*\cohomorf2{X,\mu_n}$, $q^*\cohomorf2{Y,\mu_n}$ and $\image{\epsilon}$. Therefore, we only need to claim that $z$ is orthogonal to $\image{\epsilon}$.

	The local cup-product couples for $v\in\Omega_k$,
	\begin{equation}
		\cup_v:\cohomorf{1}{k_v,G_Y}\times\cohomorf{1}{k_v,\hat G_Y}\rightarrow\cohomorf2{k_v,\mu_n}
	\end{equation}give rise to the global Poitou--Tate coupling
	\begin{equation}
		\PTcouple{\,,\,}:P^1(k_v,G_Y)\times P^1(k_v,\hat G_Y)\rightarrow\bZ/n.
	\end{equation}It is a perfect coupling, moreover, we have Poitou--Tate exact sequence
		\begin{equation}\label{PTex}
			\xymatrix{\cohomorf1{k,G_Y}\ar[r]&P^1(G_Y)\ar[r]&\cohomorf1{k,\hat G_Y}^D}.
		\end{equation}See \cite[Thm. 4.10]{ADT} for the number field case and \cite[Thm. 2.7]{Izq16} for the function field case. Here for any finite group $F$, $P^1(F)$ is the restricted product of groups $\cohomorf1{k_v,F}$, for $v\in\Omega_k$, relative to sub-groups $\cohomorf1{\cO_v,F}$, where $v$ is a non-archimedean place of $k$.Take $\psi\in\hom{k}(G_X\otimes G_Y,\mu_n)$, then the Brauer--Manin coupling $\BMcouple{\epsilon(\psi),z}$ is given by the Poitou--Tate coupling $\PTcouple{(\psi_*(\cT_X),x),(\cT_Y,y)}$ by lemma \reff{arith2} where $\psi_*$ is the induced homomorphism $\cohomorf{1}{X,G_X}\rightarrow\cohomorf{1}{X,\hat G_Y}$ and $\cT_X,\cT_Y$ are universal $n$-torsors of $X,Y$ respectively. The existence of universal $n$-torsors is guaranteed by Proposition \reff{prop_exist}.
		
		For each $a\in\cohomorf1{k,\hat G_Y}$, we have $a\cup\cT_Y\in\cohomorf2{Y,\mu_n}$ and $\PTcouple{a,(\cT_Y,y)}=\BMcouple{a\cup\cT_Y,y}=0$. Hence, the element $(\cT_Y,y)$, as an element of $P^1(G_Y)$, admits a lift in $\cohomorf1{k,G_Y}$ by Poitou--Tate exact sequence \reff{PTex}. Similarly, the element $(\psi_*(\cT_X),x)$, as an element of $P^1(\hat G_Y)$, admits a lift in $\cohomorf1{k,\hat G_Y}$. Therefore, one obtains
		$\PTcouple{(\psi_*(\cT_X),x),(\cT_Y,y)}=\sumf{v\in\Omega_k}\inv_v((\psi_*(\cT_X),x)\cup_v(\cT_Y,y))=0$ by exact sequence \reff{BHN}.
	\hfill$\square$

	\vspace{3ex}
	\textbf{Acknowledgements} This work is supported by National Natural Science Foundatiotn of China No.12071448. The authors would like to thank Yang Cao for his proposal of the tool of universal $n$-torsors to obtain the main result of the paper.

	\newpage


\end{document}